\documentclass[12pt]{amsart}

\usepackage{amsfonts, amsthm, amsmath,dsfont}

\usepackage{rotating}

\usepackage{tikz}

\usepackage{graphics}

\usepackage{amssymb}

\usepackage{amscd}

\usepackage[latin2]{inputenc}

\usepackage{t1enc}

\usepackage[mathscr]{eucal}

\usepackage{indentfirst}

\usepackage{graphicx}

\usepackage{graphics}

\usepackage{pict2e}

\usepackage{mathrsfs}

\usepackage{enumerate}

\usepackage{tikz}
\usetikzlibrary{tikzmark,arrows.meta}

\usepackage[pagebackref]{hyperref}

\hypersetup{colorlinks=true}

\usepackage{cite}
\usepackage{color}
\usepackage{epic}
\usepackage{hyperref}
\usepackage{framed}
\usepackage{mathabx}
\usepackage{booktabs}
\usepackage{makecell}
\usepackage{comment}
\newcolumntype{V}{!{\vrule width 2pt}}

\numberwithin{equation}{section}
\theoremstyle{plain}

\newtheorem{theorem}{Theorem}[section]
\newtheorem{corollary}[theorem]{Corollary}

\newtheorem{conjecture}[theorem]{Conjecture}
\newtheorem{remark}[theorem]{Remark}
\newtheorem{lemma}[theorem]{Lemma}

\def\A{\mathcal{A}}

\def\U{\mathcal{U}}

\def\N{\mathbb{N}}
\def\Q{\mathbb{Q}}

\def\inv{\mathsf{inv}}

\def\id{\mathrm{id}}

\begin{document}

\title[A Catalan-tangent number identity]{An  identity relating Catalan numbers to tangent numbers with arithmetic applications}

\author[T. Zhao]{Tongyuan Zhao}
\address[Tongyuan Zhao]{College of Science, China University of Petroleum, 102249 Beijing, P.R. China}
\email{zhaotongyuan@cup.edu.cn}

\author[Z. Lin]{Zhicong Lin}
\address[Zhicong Lin]{Research Center for Mathematics and Interdisciplinary Sciences, Shandong University \& Frontiers Science Center for Nonlinear Expectations, Ministry of Education, Qingdao 266237, P.R. China}
\email{linz@sdu.edu.cn}

\author[Y. Zang]{Yongchun Zang}
\address[Yongchun Zang]{Research Center for Mathematics and Interdisciplinary Sciences, Shandong University \& Frontiers Science Center for Nonlinear Expectations, Ministry of Education, Qingdao 266237, P.R. China}
\email{202521437@mail.sdu.edu.cn}

\date{\today}

\begin{abstract} We prove a combinatorial  identity relating Catalan numbers to tangent numbers arising from the study of peak algebra that was conjectured by Aliniaeifard and Li. This identity leads to the discovery of the intriguing identity
$$
\sum_{k=0}^{n-1}{2n\choose 2k+1}2^{2n-2k}(-1)^{k}E_{2k+1}=2^{2n+1},
$$
where $E_{2k+1}$ denote the tangent numbers. Interestingly, the latter identity can be applied to prove that $(n + 1)E_{2n+1}$ is divisible by $2^{2n}$ and the quotient is an odd number, a fact whose traditional proofs  require significant calculations. Moreover, we find a natural $q$-analog of the latter identity with a combinatorial proof.
This $q$-identity can be applied to prove Foata's divisibility property  of the $q$-tangent numbers, which responds  to a problem raised by Sch\"utzenberger.
\end{abstract}


\maketitle

\section{Introduction}
A {\em set composition} (or {\em ordered set partition})  $\phi$ of $[n]:=\{1,2,\ldots,n\}$ is a list of mutually disjoint nonempty subsets $\phi_1,\phi_2,\ldots,\phi_{\ell}$ of $[n]$ whose union is $[n]$.  It is convenient  to denote  $\phi$ by $\phi_1/\phi_2/\cdots/\phi_{\ell}\vDash[n]$. Each $\phi_i$ is called a {\em block} of $\phi$ and the number of blocks of $\phi$ is denoted by $\ell(\phi)$. If each block of $\phi$ has odd size, then $\phi$ is said to be {\em odd}.
 For instance, $\phi=6/148/9/235/7$ is an odd set composition of $[9]$ with $\ell(\phi)=5$. Odd set compositions were considered by Aliniaeifard and  Li~\cite{AL} in studying  the internal coproduct formula for peak algebra.

The {\em Catalan numbers} 
$$\left\{C_n=\frac{1}{n+1}{2n\choose n}\right\}_{n\geq0}=\{1,1,2,5,14,42,132,\ldots\}$$
and the {\em Euler numbers} $E_n$, defined by the Taylor expansion
$$
\tan(x)+\sec(x)=\sum_{n\geq1}E_n\frac{x^n}{n!}=1+x+1\frac{x^2}{2!}+2\frac{x^3}{3!}+5\frac{x^4}{4!}+16\frac{x^5}{5!}+61\frac{x^6}{6!}+\cdots,
$$
are two of the most important combinatorial sequences in enumerative combinatorics.
The numbers $E_n$ with even indices and odd indices are also known as {\em secant numbers} and
{\em tangent numbers}, respectively. A classical result of Andr\'e~\cite{An} asserts that $E_n$ enumerates permutations $\pi$ of $[n]$ that satisfy the {\em alternating} property:
$$
\pi_1>\pi_2<\pi_3>\pi_4<\cdots.
$$
Since the discovery of Andr\'e, the Euler numbers have been investigated extensively in combinatorics; see the survey of Stanley~\cite{St10}. In particular, several combinatorial identities arising from permanents for Euler numbers involving Stirling numbers of the second kind were found in~\cite{LYZ,FLS}.
 Recently, in their  study of non-commutative peak algebra, Aliniaeifard and  Li~\cite{AL} found the following combinatorial identity linking Catalan numbers with tangent numbers.

\begin{conjecture}[\text{Aliniaeifard and  Li~\cite[Conjecture~11.7]{AL}}]\label{conj:AL}
For even positive integer $\ell$, let $\mu_{\ell}=(-1)^{\frac{\ell}{2}-1}C_{\frac{\ell}{2}-1}$. For any odd positive integer $n$, we have
\begin{equation}\label{eq:ALconj}
\sum_{\phi\vDash[n]\atop\text{$\phi$ odd}}2^{n-\ell(\phi)}\mu_{\ell(\phi)+1}=(-1)^{\frac{n-1}{2}}E_n.
\end{equation}
\end{conjecture}

For instance, since all the odd set compositions of $[3]$ are 
$$
123, 1/2/3, 1/3/2, 2/1/3, 2/3/1, 3/1/2, 3/2/1,
$$
identity~\eqref{eq:ALconj} reads 
$$
2^{3-1}(-1)^0C_0+6(-1)^1C_1=(-1)^1E_3.
$$
\begin{remark}
The signed Catalan numbers $\mu_{\ell}$ are the M\"obius functions of the posets of all odd set compositions with refinement order (see~\cite{AL,CP,SS}). Originally, \eqref{eq:ALconj} is stated as
\begin{equation}\label{eq:ALconj2}
\sum_{\phi\vDash[n]\atop\text{$\phi$ odd}}2^{n-\ell(\phi)}\mu_{\ell(\phi)+1}=\sum_{0\leq k\leq n-1}(-1)^k A(n,k),
\end{equation}
where $A(n,k)$ is the Eulerian number that enumerates permutations of $[n]$ with $k$ descents. As  was known (see~\cite{LZ}) that for odd $n$
$$
\sum_{0\leq k\leq n-1}(-1)^k A(n,k)=(-1)^{\frac{n-1}{2}}E_n,
$$
\eqref{eq:ALconj2} and~\eqref{eq:ALconj} are equivalent.
\end{remark}

The main objective of this paper is to prove Conjecture~\ref{conj:AL} using generating functions and to present some of its interesting consequences and  applications.

An  immediate consequence of~\eqref{eq:ALconj} is the discovery of the following intriguing  combinatorial identity for tangent numbers.
\begin{corollary}\label{cor:genocc}
For $n\geq1$,  we have
\begin{equation}\label{eq:Genocchi}
\sum_{k=0}^{n-1}{2n\choose 2k+1}2^{2n-2k}(-1)^{k}E_{2k+1}=2^{2n+1}.
\end{equation}
\end{corollary}

\begin{proof}
As  observed by Aliniaeifard and  Li~\cite{AL}, the identity~\eqref{eq:Genocchi}
follows by combining~\cite[Theorem~11.5]{AL} and Conjecture~\ref{conj:AL}.
\end{proof}

The identity~\eqref{eq:Genocchi} is neat and deserves a pure combinatorial proof. We find such a proof which leads to a $q$-analog of~\eqref{eq:Genocchi} (see Theorem~\ref{thm:qgeno}) using a natural $q$-analog of the tangent numbers studied in~\cite{AG,Fo, FH10,LZ,St0,St10}.
As it turns out, our $q$-analog of~\eqref{eq:Genocchi} can be applied to prove Foata's divisibility property   of the $q$-tangent numbers~\cite{Fo}, which responds  to a problem raised by Sch\"utzenberger.

The tangent number $E_{2n+1}$ enumerates  complete increasing binary trees with $2n + 1$ nodes (see~\cite{St10}). This combinatorial interpretation immediately proves that $E_{2n+1}$ is divisible by $2^n$.   However, a stronger divisibility property is known in relation to Bernoulli and Genocchi numbers, namely, the divisibility of $(n + 1)E_{2n+1}$ by $2^{2n}$. The traditional proofs of this fact require significant calculations (see~\cite{Ba,Ca1,Ca2,FH,JS}) and a combinatorial proof was provided by Han and Liu~\cite{HL}  using increasing labelled binary trees. Surprisingly,  identity~\eqref{eq:Genocchi} leads to a simple proof of this divisibility.

\begin{corollary}\label{cor:divisi}
 The integer  $(n + 1)E_{2n+1}$ is divisible by $2^{2n}$, and the quotient is an odd number.
\end{corollary}

The quotients
$$
\frac{(n + 1)E_{2n+1}}{2^{2n}}=:G_{2n+2}
$$
are known as the {\em Genocchi numbers}, whose generating function is
$$
\sum_{n\geq0} G_{2n+2}\frac{x^{2n+2}}{(2n+2)!}=x\tan(x/2)=\frac{x^2}{2!}+\frac{x^4}{4!}+3\frac{x^6}{6!}+17\frac{x^8}{8!}+155\frac{x^{10}}{10!}+\cdots.
$$
The fact that Genocchi numbers $G_{2n+2}$ are odd integers is traditionally proven  using the von Staudt--Clausen theorem on Bernoulli numbers and the little Fermat theorem~\cite{Ca1,Ca2,JS}.
A different proof by using the Laplace transform was given by Barsky~\cite{Ba} (see also~\cite[Chapitre~3]{FH}). A combinatorial proof using the hook length formula for binary trees was provided by Han and Liu~\cite{HL}.  Our proof of Corollary~\ref{cor:divisi} is elementary, using only  identity~\eqref{eq:Genocchi} via simple induction.

The rest of this paper is organized as follows. In Section~\ref{sec:1}, we present a $q$-analog of~\eqref{eq:Genocchi} with a combinatorial proof.
This $q$-analog of~\eqref{eq:Genocchi} is then applied in Section~\ref{sec:foata} to provide an alternative proof of  Foata's divisibility property   of the $q$-tangent numbers.
The proofs of Corollary~\ref{cor:divisi} and Conjecture~\ref{conj:AL} are given in Sections~\ref{sec:2} and~\ref{sec:3}, respectively. Finally, this paper concludes with some  remarks in Section~\ref{sec:4}.

\section{Combinatorial proof of a $q$-analog of~\eqref{eq:Genocchi} }
\label{sec:1}
In this section, we find a combinatorial proof of~\eqref{eq:Genocchi}, leading to its $q$-analog.

We begin with some $q$-notations. The {\em$q$-shifted factorial} $(t;q)_n$ is defined by
$(t;q)_n :=\prod_{i=0}^{n-1}(1-tq^i)$ for any positive integer $n$ and $(t;q)_0=1$. In particular, $$(-q;q)_n=(1+q)(1+q^2)\cdots(1+q^n)
$$ is a natural $q$-analog of $2^n$. Recall that the $q$-binomial numbers $\left[{n\atop k}\right]_q$ are defined by
$$
{n\brack k}_q :=\frac{(q;q)_n }{(q;q)_{n-k}(q;q)_k}\qquad \textrm{for}\quad 0\leq k\leq n,
$$
and ${n\brack k}_q=0$ if $k<0$ or $k>n$.
 The {\em$q$-sine} $\sin_q(x)$ and {\em$q$-cosine} $\cos_q(x)$ are defined by
$$
\sin_q(x):=\sum_{n\geq0}(-1)^n\frac{x^{2n+1}}{(q;q)_{2n+1}} \quad\text{and}\quad\cos_q(x):=\sum_{n\geq0}(-1)^n\frac{x^{2n}}{(q;q)_{2n}}.
$$
The concerned {\em$q$-tangent numbers} $E_{2n+1}(q)$  then occur  in the  expansion of $\tan_q(x)$:
$$
\tan_q(x):=\frac{\sin_q(x)}{\cos_q(x)}=\sum_{n\geq0}E_{2n+1}(q)\frac{x^{2n+1}}{(q;q)_{2n+1}}.
$$
Let $\A_n$ be the set of all alternating (or down-up) permutations of $[n]$. It is a classical result (see~\cite{FH10,St10}) that $E_{2n+1}(q)$ has the combinatorial  interpretation
\begin{equation}\label{int:qtan}
E_{2n+1}(q)=\sum_{\pi\in\A_{2n+1}}q^{\inv(\pi)},
\end{equation}
where $\inv(\pi):=|\{(i,j)\in[n]^2: i<j,\pi_i>\pi_j\}|$ is the {\em inversion number} of a permutation $\pi$.
The main result in this section is the following $q$-analog of~\eqref{eq:Genocchi}.

\begin{theorem}
\label{thm:qgeno}
For $n\geq1$,  we have
\begin{equation}\label{eq:qGeno}
\sum_{k=0}^{n-1}{2n\brack 2k+1}_q(-q;q)_{2n-2k-2}(-1)^{k}E_{2k+1}(q)=(-q;q)_{2n-1}.
\end{equation}
\end{theorem}

We need some preparations before we can prove Theorem~\ref{thm:qgeno}.
A word $w=w_1w_2\cdots w_k$ on $\N$ with distinct letters is called {\em unimodal} if
$$
w_1<w_2<\cdots<w_{\ell}>w_{\ell+1}>\cdots>w_k
$$
for some $1\leq\ell\leq k$. Let $\U_n$ be the set of all unimodal permutations of $[n]$. For instance,
$\U_3=\{123,132,231,321\}$. It is clear that $|\U_n|=2^{n-1}$. We require  the following lemma, which  enumerates unimodal permutations by  inversion numbers.

\begin{lemma}\label{lem:unimodal}
For $n\geq1$, we have
\begin{equation}
\sum_{\pi\in\U_n}q^{\inv(\pi)}=(-q;q)_{n-1}.
\end{equation}
\end{lemma}

\begin{proof}
It is known~\cite[Proposition~1.3.17]{st0} that  the $q$-binomial coefficient
has the combinatorial interpretation
\begin{equation}\label{eq:qmul}
{n\brack  k}_{q}=\sum_{({\mathcal A}, {\mathcal B})}q^{\inv({\mathcal A}, {\mathcal B})},
\end{equation}
where the sum is over all set compositions $({\mathcal A}, {\mathcal B})$ of $[n]$ such that $|{\mathcal A}|=k$ and
$$\inv({\mathcal A}, {\mathcal B}):=|\{(a,b)\in{\mathcal A}\times{\mathcal B}: a>b\}|.$$ Each permutation  $\pi\in\U_n$ with $\pi_{n-k}=n$ (for some $0\leq k\leq n-1$) can be decomposed as
$$
\pi=\pi_1\pi_2\cdots\pi_{n-k-1},n,\pi_{n-k+1}\pi_{n-k+2}\cdots\pi_n,
$$
where $\pi_1<\pi_2<\cdots<\pi_{n-k-1}$ and $\pi_{n-k+1}>\pi_{n-k+2}>\cdots>\pi_n$. Since
$$
\inv(\pi)=\inv(\{\pi_1,\pi_2,\ldots,\pi_{n-k-1}\},\{\pi_{n-k+1},\pi_{n-k+2},\ldots,\pi_n\})+{k+1\choose 2},
$$
it follows from~\eqref{eq:qmul} that
\begin{equation}
\sum_{\pi\in\U_n}q^{\inv(\pi)}=\sum_{k=0}^{n-1}q^{{k+1\choose 2}}{n-1\brack k}_q=(-q;q)_{n-1},
\end{equation}
where the last equality is obtained by applying the {\em$q$-binomial theorem}
$$
\prod_{j=1}^n(1+tq^j)=\sum_{k=0}^nq^{{k+1\choose 2}}{n\brack k}_qt^k.
$$
This completes the proof of the lemma.
\end{proof}

We are ready for the proof of Theorem~\ref{thm:qgeno}.

\begin{proof}[{\bf Combinatorial proof of Theorem~\ref{thm:qgeno}}]
Consider the set $S_k$ of all permutations $\pi=\pi_1\pi_2\ldots\pi_{2n}$ of $[2n]$ such that
$$
\pi_1\pi_2\cdots\pi_{2n-2k-1}\text{ is unimodal and  } \pi_{2n-2k}\pi_{2n-2k+1}\cdots\pi_{2n} \text{ is alternating}.
$$
Let $T_k$ be the subset of $S_k$ consisting of these permutations satisfying
$$\pi_{2n-2k-1}<\pi_{2n-2k}\text{ and the prefix } \pi_1\pi_2\cdots\pi_{2n-2k-1}\text{ is not increasing}.$$
Then, we have  $T_{i-1}=S_i-T_i$ for $1\leq i\leq n-1$. Thus,
\begin{equation}\label{eq:incluexclu}
\U_{2n}=S_0-T_0=S_0-(S_1-T_1)=\cdots=S_0-S_1+S_2-S_3+\cdots,
\end{equation}
where ``$-$'' is set subtraction and ``$+$'' is set addition. By Lemma~\ref{lem:unimodal}, the interpretation~\eqref{eq:qmul} of the $q$-binomial coefficients and the interpretation~\eqref{int:qtan} of the $q$-tangent numbers,   we have
$$
\sum_{\pi\in S_k}q^{\inv(q)}={2n\brack 2k+1}_q(-q;q)_{2n-2k-2}E_{2k+1}(q).
$$
It then follows from~\eqref{eq:incluexclu} that
$$
\sum_{\pi\in\U_{2n}}q^{\inv(\pi)}=\sum_{k=0}^{n-1}{2n\brack 2k+1}_q(-q;q)_{2n-2k-2}(-1)^{k}E_{2k+1}(q),
$$
which completes the proof of the theorem in view of Lemma~\ref{lem:unimodal}.
\end{proof}

\section{On Foata's divisibility property   of the $q$-tangent numbers}
\label{sec:foata}
In this section, we provide an alternative proof of  Foata's divisibility property  of the $q$-tangent numbers  using the identity~\eqref{eq:qGeno}.

Since by Corollary~\ref{cor:divisi} we have
$$
(n + 1)E_{2n+1}=2^{2n}G_{2n+2}
$$
with $G_{2n+2}$ an odd integer, 
Sch\"utzenberger (see~\cite{Fo}) raised the problem of finding a polynomial of the form 
$$
\prod_{i\geq1}(1+q^i)^{a(n,i)}
$$
that divides $E_{2n+1}(q)$. In particular, he made a conjecture, which asserts that $E_{2n+1}(q)$  is divisible by $(-q;q)_n$, at the combinatorics conference at Oberwohlfach in 1975.  His conjecture was confirmed by Andrews and Gessel~\cite{AG} who showed that $E_{2n+1}(q)$  is divisible by 
$(1+q)^{\lfloor n/2\rfloor}(-q;q)_n$. Foata~\cite{Fo} further strengthened the latter result as follows.

Every integer $n$ may be written uniquely as $n = m2^l$ with $l \geq 0$ and $m$ odd. 
Following Foata~\cite{Fo}, we  can define the polynomial
\[
Ev_n(q) = \prod_{j=0}^{l} \left(1 + q^{m2^j}\right).
\]
For $n\geq1$,  let
\[
D_n(q):=
\begin{cases}
\prod_{i=1}^{n} Ev_i(q), & n \text{ odd}, \\
(1 + q^2) \prod_{i=1}^{n} Ev_i(q), & n \text{ even}.
\end{cases}
\]
For the sake of convenience, let $Ev_0(q) = D_0(q) = 1$. With the recursion
\begin{equation}\label{rec:Foa}
E_{2n + 1}(q) = \sum_{k=0}^{n - 1} {2n\brack 2k+1}_q q^{2k + 1} E_{2k + 1}(q) E_{2n - 2k - 1}(q) \quad (n \geq 1)
\end{equation}
and analysis of cyclotomic polynomials, Foata~\cite[Theorem~1]{Fo} proved the following divisibility property.

\begin{theorem}[Foata] \label{thm:Fo}
For each $n\geq 1$,    $E_{2n-1}(q)$ is divisible by $D_{n-1}(q)$.
\end{theorem}

For the reader's convenience, the first few factorizations (in the ring $\Q[q]$) of $E_{2n-1}(q)$ are listed below:
\begin{align*}
  E_1(q) &= 1 ,\\
  E_3(q) & = (1 + q)q ,\\
  E_5(q) & = (1 + q)^2 (1 + q^2)^2 q^2 ,\\
  E_7(q) &= (1 + q)^2  (1 + q^2) (1 + q^3)q^3  (1 + q + 3q^2 + 2q^3 + 3q^4 + 2q^5 + 3q^6 + q^7 + q^8),\\
  E_9(q) &= (1 + q)^3  (1 + q^2)^3 (1 + q^3)(1 + q^4)\times\\& \,\quad q^4  (1 + q + 3q^2 + 2q^3 + 3q^4 + 3q^5 + 5q^6 + 3q^7 + 3q^8 + 2q^9 + 3q^{10} + q^{11} + q^{12}).
 \end{align*}
The purpose of this section is to provide an alternative approach to  Theorem~\ref{thm:Fo} via  the identity~\eqref{eq:qGeno}. 

\begin{remark}
Setting $q=1$ in~\eqref{rec:Foa} gives 
$$
E_{2n + 1} = \sum_{k=0}^{n - 1} {2n\choose 2k+1} E_{2k + 1} E_{2n - 2k - 1} \quad (n \geq 1). 
$$
Unlike our identity~\eqref{eq:Genocchi}, we have no idea how the above recursion  can be applied to prove Corollary~\ref{cor:divisi}.
\end{remark}

We need two auxiliary lemmas before we  prove  Theorem~\ref{thm:Fo}. 
For two polynomials $f (q), g(q) \in \mathbb{Z}[q]$, it is convenient to write  $f (q) \mid g(q)$ if
 $g(q)/f(q) \in \mathbb{Z}[q]$. 
Following Lin et al.~\cite{LM}, we introduce 
$$
B_n(q):=\prod_{k\geq1}\prod_{i=1}^{\lfloor n/2^k \rfloor}(1+q^i).
$$
It was  observed in~\cite[Remark~4.4]{LM} that 
\begin{equation}\label{eq:LinFo}
B_{2n}(q)=B_{2n+1}(q)=\prod_{i=1}^{n} Ev_i(q).
\end{equation}
The following characterization of the factor $B_n(q)$ in polynomials proved in~\cite[Theorem~4.1]{LM} is useful in our approach.

\begin{lemma} \label{thm:LWMW}
For a fixed integer $n\geq1$ and a polynomial $f(q) \in \mathbb{Z}[q]$,
$$
B_n(q) \mid f(q) \Longleftrightarrow  (1 + q^m)^{\lfloor n/ 2m \rfloor}\mid f(q)
\enspace \text{for all} \enspace 1 \leq m \leq \lfloor n/ 2 \rfloor.
$$
\end{lemma}

Using Lemma~\ref{thm:LWMW}, we can prove the following  result.

\begin{lemma} \label{lem:GQ}

For $n\geq1$, we have $B_{2n}(q)\mid(-q;q)_{2n-1}$.
\end{lemma}

\begin{proof}
According to Lemma~\ref{thm:LWMW}, we only need to show that for $1 \leq m \leq \lfloor n/ 2 \rfloor$,
\begin{equation}\label{eq:GQ}
(1 + q^m)^{\lfloor 2n/ 2m \rfloor} \mid (-q;q)_{2n-1} = \prod_{i=1}^{2n-1}(1+q^i).
\end{equation}
Since $(1 + q^m)\mid(1 + q^n)$ if and only if $n/m$ is an odd integer, ~\eqref{eq:GQ} is derived from
$$
 \lfloor n/m \rfloor \leq \left\lfloor \biggl(\frac{2n-1}{m}+1\biggr)/2\right\rfloor,
$$
which completes the proof.
\end{proof}

Now we are ready for the proof of Theorem~\ref{thm:Fo}.

\begin{proof}
Recall that 
$$
B_{2n}(q)=\prod_{i=1}^{n} Ev_i(q).
$$
We first prove that $B_{2n-2}(q)\mid E_{2n-1}(q)$ using the identity~\eqref{eq:qGeno}. We proceed by induction on $n$.
By the induction hypothesis, we have
\begin{equation}\label{eq:Ev|E}
Ev_0(q) \cdots Ev_{k-1}(q) \mid E_{2k-1}(q) \quad (1 \leq k\leq  n - 1).
\end{equation}
 According to ~\cite[Lemma~2.2]{Fo}, there holds
$$
Ev_{n-k+1}(q) \cdots Ev_{n}(q) ~\bigg|~ {2n\brack 2k-1}_q Ev_{1}(q) \cdots Ev_{k-1}(q).
$$
Replacing $k$ by $n-k+1$, we have
\begin{equation}\label{eq:Fo2}
Ev_{k}(q) \cdots Ev_{n}(q) ~\bigg|~ {2n\brack 2k-1}_q Ev_{1}(q) \cdots Ev_{n-k}(q).
\end{equation}
Combining~\eqref{eq:LinFo},~\eqref{eq:Ev|E} and~\eqref{eq:Fo2} gives
$B_{2n}(q) \mid {2n\brack 2k-1}_q B_{2n-2k}(q)E_{2k-1}(q)$.
Applying Lemma~\ref{lem:GQ}, we have $B_{2n-2k}(q) \mid (-q;q)_{2n-2k}$.
Thus,
$$
B_{2n}(q) ~\bigg|~ {2n\brack 2k-1}_q (-q;q)_{2n-2k} E_{2k-1}(q).
$$
Now we have shown that each term of~\eqref{eq:qGeno} is divisible by $B_{2n}(q)$
except
$$
{2n\brack 2n-1}_q (-q;q)_{0}(-1)^{n-1} E_{2n-1}(q),
$$
so $B_{2n}(q)=B_{2n-2}(q)Ev_{n}(q) \mid (1+q+ \cdots q^{2n-1}) E_{2n-1}(q)$.
Assume that $n=m 2^l$ where $l\geq0$ and $m$ odd. Notice that
$$
\frac{1+q+ \cdots q^{2n-1}}{Ev_n(q)}=\frac{1-q^{m2^{l+1}}}{1-q^m}\cdot\frac{1-q^m}{1-q}\cdot\frac{1}{\prod_{j=0}^{l} (1 + q^{m2^j})}=\frac{1-q^m}{1-q},
$$
and  $B_{2n-2}(q), \frac{1-q^m}{1-q}$ share no common roots (since $1+q^i, 1-q^m$ share no common roots for $i\geq1$),
we have $B_{2n-2}(q)\mid E_{2n-1}(q)$. This proves that $E_{2n-1}(q)$ is divisible by $D_{n-1}(q)$ when $n$ is even.

For $n$ odd, we need to prove that $D_{n-1}(q)=(1+q^2)B_{2n-2}(q)\mid E_{2n-1}(q)$. 
Denote by $V_{f}(g)$ the
exponent of the polynomial factor $f(q)$ in $g(q)$, that is,
$$f(q)^{V_{f}(g)}\mid g(q) \text{ and } f(q)^{V_{f}(g)+1}\nmid g(q).
$$
Since $1+q^2$ and $1+q^i$
share no common roots where $i$ is odd or $4\mid i$, there holds 
$$
V_{1+q^2}(Ev_{i}(q))=
\begin{cases}
0,\quad \text{ $i$ odd},\\
1, \quad \text{  $i$ even}.
\end{cases}
$$
 Thus, 
$$
V_{1+q^2}(B_{2n-2}(q))=\sum_{i\leq n-1\atop{i \text{ even}}}1=\frac{n-1}{2},
$$
and we only need to prove that
$$
V_{1+q^2}(E_{2n-1}(q))\geq V_{1+q^2}\left((1+q^2)B_{2n-2}(q)\right) =\frac{n+1}{2}.
$$
We proceed to prove this by induction on $n$.

Similar to $V_{1+q^2}(B_{2n-2}(q))=\lfloor(n-1)/2\rfloor$, one can verify that $V_{1+q^2}((-q;q)_t)=\lfloor(t+2)/4\rfloor$ for $t\geq0$.
Thus, for $0\leq k \leq n-1$, we have
$$
V_{1+q^2}\left({2n\brack 2k+1}_q\right)
=\left\lfloor\frac{2n}{4}\right\rfloor-\left\lfloor\frac{2k+1}{4}\right\rfloor-\left\lfloor\frac{2n-2k-1}{4}\right\rfloor=
\begin{cases}
1, \quad& k \text{ odd}, \\
0, & k \text{ even}.
\end{cases}
$$
By the induction hypothesis, for $0\leq k \leq n-2$,
$$
V_{1+q^2}\left(E_{2k+1}(q)\right)\geq
\begin{cases}
(k-1)/2, & k \text{ odd}, \\
(k+2)/2, & k \text{ even}.
\end{cases}
$$
Then
\begin{align*}
&\quad \,\,  V_{1+q^2}\biggl({2n\brack 2k+1}_q(-q;q)_{2n-2k-2}E_{2k+1}(q)\biggr)\\
&= V_{1+q^2}\left({2n\brack 2k+1}_q\right)
+ \left\lfloor\frac{2n-2k-2+2}{4}\right\rfloor+V_{1+q^2}\left(E_{2k+1}(q)\right)\\
& \geq
\begin{cases}
1+(n-k)/2+(k-1)/2=(n+1)/2, & k \text{ odd}, \\
(n-k-1)/2+(k+2)/2=(n+1)/2, & k \text{ even}.
\end{cases}
\end{align*}
Thus, for $1\leq k \leq n-2$
$$
(1+q^2)^{(n+1)/2} ~\bigg|~ {2n\brack 2k+1}_q(-q;q)_{2n-2k-2}E_{2k+1}(q).
$$
Rewrite~\eqref{eq:qGeno} as 
 \begin{equation}\label{eq:qGeno2}
\sum_{k=1}^{n-1}{2n\brack 2k+1}_q(-q;q)_{2n-2k-2}(-1)^{k}E_{2k+1}(q)=\Delta(q),
\end{equation}
where
$\Delta(q):=(-q;q)_{2n-1}-{2n\brack 1}_q(-q;q)_{2n-2}$.
Notice that
\begin{align*}
 V_{1+q^2}(\Delta(q))&= V_{1+q^2}\left((-q;q)_{2n-2}\right)+V_{1+q^2}\left(1+q+\cdots+q^{2n-1}-(1+q^{2n-1})\right)\\
&=\left\lfloor\frac{2n-2+2}{4}\right\rfloor+V_{1+q^2}\left(q(1+q)(1+q^2)\sum_{i=0}^{(n-3)/2}q^{4i}\right)\\
&=\frac{n-1}{2}+1=\frac{n+1}{2},
\end{align*}
we have that each term in~\eqref{eq:qGeno} is divisible by $(1+q^2)^{(n+1)/2}$
except
$$
{2n\brack 2n-1}_q (-q;q)_{0}(-1)^{n-1} E_{2n-1}(q).
$$
Since $V_{1+q^2}\left({2n\brack 2n-1}_q\right)=0$, we have
$$
V_{1+q^2}\left(E_{2n-1}(q)\right)=V_{1+q^2}\biggl({2n\brack 2n-1}_q E_{2n-1}(q)\biggr)\geq\frac{n+1}{2},
$$
as desired. 
This completes the proof of the theorem.
\end{proof}
\section{Proof of Corollary~\ref{cor:divisi}}
\label{sec:2}
This section aims to prove Corollary~\ref{cor:divisi}  by using identity~\eqref{eq:Genocchi}.
\begin{proof}[{\bf Proof of Corollary~\ref{cor:divisi}}]
 Rewrite~\eqref{eq:Genocchi} as
\begin{equation}\label{eq:Genocchi2}
nE_{2n-1}=2^{2n-2}-\sum_{k=0}^{n-2}{2n\choose 2k+1}2^{2n-2k-3}(-1)^{k}E_{2k+1}.
\end{equation}
We proceed  to prove the result
by induction on $n$ basing on~\eqref{eq:Genocchi2}.
Observe that
\begin{equation}\label{eq:guo}
{2n\choose 2k+1}=\frac{2k+2}{2n+1}{2n+1\choose 2k+2}.
\end{equation}
Since $2n+1$ is odd and ${2n+1\choose 2k+2}$ is a natural number, the power of $2$ in $(2k+2)$ is not greater than that in ${2n\choose 2k+1}$. Thus, for $0\leq k\leq n-2$, as $2(k+1)E_{2k+1}$ is divisible by $2^{2k+1}$ by the induction hypothesis, ${2n\choose 2k+1}2^{2n-2k-3}(-1)^{k}E_{2k+1}$ is divisible by $2^{2n-2}$. It then follows from the recursion~\eqref{eq:Genocchi2}  that $nE_{2n-1}$ is divisible by $2^{2n-2}$. This proves that $(n + 1)E_{2n+1}$ is divisible by $2^{2n}$ by induction.

It remains to show that $\frac{(n + 1)E_{2n+1}}{2^{2n}}=G_{2n+2}$ is odd. Dividing both sides of~\eqref{eq:Genocchi2} by $2^{2n-2}$ and then using~\eqref{eq:guo} we get
\begin{equation}\label{eq:Genocchi3}
G_{2n}=1-\sum_{k=0}^{n-2}\frac{1}{2n+1}{2n+1\choose 2k+2}(-1)^{k}G_{2k+2}.
\end{equation}
 We again proceed
by induction on $n$. By the induction hypothesis, we can assume that $G_{2k+2}$ are odd for all $0\leq k\leq n-2$. Thus, the integer $\frac{1}{2n+1}{2n+1\choose 2k+2}(-1)^{k}G_{2k+2}$ has the same parity as ${2n+1\choose 2k+2}$ for $0\leq k\leq n-2$. It follows from~\eqref{eq:Genocchi3} that the parity of $G_{2n}$ is the same as that of
$$
1+\sum_{k=0}^{n-2}{2n+1\choose 2k+2}=1+\sum_{k=2}^{n}{2n+1\choose k}=2^{2n}-(2n+1),
$$
which is obviously odd. This completes the proof of Corollary~\ref{cor:divisi} by induction.
\end{proof}

\section{Proof of Conjecture~\ref{conj:AL}}
\label{sec:3}
Recall the three hyperbolic functions
$$
\sinh(x)=\frac{e^x-e^{-x}}{2}, \quad\cosh(x)=\frac{e^x+e^{-x}}{2}\quad\text{and}\quad\tanh(x)=\frac{\sinh(x)}{\cosh(x)}.
$$
We need the following lemma in our generating function proof of Conjecture~\ref{conj:AL}. 
\begin{lemma}\label{lem:1}
Let $a_{n,k}=|\{\phi\vDash[n]: \text{$\phi$ odd}, \ell(\phi)=k\}|$.
Then
$$
\sum_{n,k} a_{n,k}y^k\frac{x^n}{n!}=\sum_{k\geq1}(y\sinh(x))^k=\frac{y\sinh(x)}{1-y\sinh(x)}.
$$
\end{lemma}

\begin{proof}
This is an easy application of the compositional formula for the exponential generating functions (see~\cite[Proposition~5.1.3]{St}).
\end{proof}
We can now prove Conjecture~\ref{conj:AL}.
\begin{proof}[{\bf Proof of Conjecture~\ref{conj:AL}}]
By Lemma~\ref{lem:1}, the exponential generating function for the left-hand side of~\eqref{eq:ALconj} is
\begin{align*}
\sum_{\text{$n$ odd}}\frac{x^n}{n!}\sum_{\phi\vDash[n]\atop\text{$\phi$ odd}}2^{n-\ell(\phi)}\mu_{\ell(\phi)+1}&=\sum_{\text{$n$ odd}}\frac{x^n}{n!}\sum_{\text{$k$ odd}}2^{n-k}a_{n,k}\mu_{k+1}\\
&=\sum_{\text{$n$ odd}}\frac{x^n}{n!}\sum_{\text{$k$ odd}}2^{n-k}n![x^n](\sinh(x))^k\mu_{k+1}\\
&=\sum_{\text{$n$ odd}} (2x)^n\sum_{k\geq0}[x^n]\biggl(\frac{\sinh(x)}{2}\biggr)^{2k+1}(-1)^kC_k\\
&=\sum_{k\geq0}\biggl(\frac{\sinh(2x)}{2}\biggr)^{2k+1}(-1)^kC_k\\
&=\frac{\sinh(2x)}{2}\sum_{k\geq0}\biggl(\frac{-\sinh^2(2x)}{4}\biggr)^kC_k\\
&=\frac{\cosh(2x)-1}{\sinh(2x)}\\
&=\tanh(x),
\end{align*}
where the second last equality follows from the ordinary generating function for Catalan numbers
$$
\sum_{k\geq0}C_kx^k=\frac{1-\sqrt{1-4x}}{2x}.
$$
Thus, we have
\begin{equation}\label{eq:tanh}
\sum_{\text{$n$ odd}}\frac{x^n}{n!}\sum_{\phi\vDash[n]\atop\text{$\phi$ odd}}2^{n-\ell(\phi)}\mu_{\ell(\phi)+1}=\tanh(x).
\end{equation}
On the other hand, we have
$$
\sum_{\text{$n$ odd}}(-1)^{\frac{n-1}{2}}E_n\frac{x^n}{n!}=\frac{\tan(ix)}{i}=\frac{\sinh(x)}{\cosh(x)}=\tanh(x).
$$
Comparing with~\eqref{eq:tanh} proves Conjecture~\ref{conj:AL}.
\end{proof}

\section{Concluding  remarks}
\label{sec:4}

In this work, we prove two combinatorial identities,~\eqref{eq:ALconj} and~\eqref{eq:Genocchi},  involving tangent numbers found in the work of Aliniaeifard and  Li~\cite{AL} with arithmetic applications. One may wonder whether there exist any similar  combinatorial identities for secant numbers.
For an analog of~\eqref{eq:ALconj} for secant numbers, consider even set compositions of $[2n]$, where a set composition  is called {\em even} if  each of its blocks has even size. Sagan~\cite{Sa} proved combinatorially using a simple sign-reversing involution that
\begin{equation}
\sum_{\phi\vDash[n]\atop\text{$\phi$ even}}(-1)^{\ell(\phi)}=(-1)^{\frac{n}{2}}E_n
\end{equation}
for $n$ being even. However, we have thus far been unable to find an involution proof of~\eqref{eq:ALconj}. Our combinatorial proof of~\eqref{eq:Genocchi} may shed some light on finding such a proof.

Regarding~\eqref{eq:Genocchi},  we have the following similar identity for  the secant numbers
\begin{equation} \label{eq:Secant}
\sum_{k=0}^{n}{2n+1\choose 2k}2^{2n-2k}(-1)^{k}E_{2k}=1.
\end{equation}
This identity can be proved combinatorially using similar ideas in our combinatorial proof of~\eqref{eq:Genocchi}. For $0\leq k\leq n$, let us consider the set $V_k$ of all permutations $\pi=\pi_1\pi_2\ldots\pi_{2n+1}$ of $[2n+1]$ such that
$$
\pi_1\pi_2\cdots\pi_{2n-2k+1}\text{ is unimodal and  } \pi_{2n-2k+2}\pi_{2n-2k+3}\cdots\pi_{2n+1} \text{ is alternating}.
$$
For $1\leq k\leq n$, let $W_k$ be the subset of $V_k$ consisting of these permutations satisfying
$$\pi_{2n-2k+1}<\pi_{2n-2k+2}\text{ and the prefix } \pi_1\pi_2\cdots\pi_{2n-2k+1}\text{ is not increasing}.$$
Let $\id_{2n+1}$ be the identity permutation of $[2n+1]$.
Then, we have $V_0-\{\id_{2n+1}\}=V_1-W_1$ and  $W_{i}=V_{i+1}-W_{i+1}$ for $1\leq i\leq n-1$. Thus,
\begin{equation}\label{eq:comsec}
\{\id_{2n+1}\}=V_0-V_1+W_1=V_0-V_1+(V_2-W_2)=\cdots=V_0-V_1+V_2-V_3+\cdots,
\end{equation}
from which~\eqref{eq:Secant} follows. Taking the inversion numbers into account and using Lemma~\ref{lem:unimodal}, we get the following $q$-analog of~\eqref{eq:Secant}.

\begin{theorem}
\label{thm:qsec}
For $n\geq0$,  we have
\begin{equation}\label{eq:qsec}
\sum_{k=0}^{n}{2n+1\brack 2k}_q(-q;q)_{2n-2k}(-1)^{k}E_{2k}(q)=1,
\end{equation}
where
\begin{equation*}
E_{2k}(q):=\sum_{\pi\in\A_{2k}}q^{\inv(\pi)}.
\end{equation*}
\end{theorem}

Finally, it would be interesting to find $q$-analogs of~\eqref{eq:ALconj}.

\section*{Acknowledgement} 
We are indebted to Victor J.W. Guo for his observation of~\eqref{eq:guo} that leads to our proof of Corollary~\ref{cor:divisi}.  We also thank Guo-Niu Han and Shu Xiao Li for their helpful discussions. This work was supported by the National  Science Foundation of China (grants 12322115, 12271301 \& 12201641)
 and the Fundamental Research Funds for the Central Universities.

\end{document}